\documentclass[11pt]{article}

\usepackage[a4paper,margin=1in]{geometry}
\usepackage{amsmath,amssymb,amsthm,mathtools}
\usepackage{microtype}
\usepackage[hidelinks]{hyperref}
\allowdisplaybreaks

\newtheorem{theorem}{Theorem}[section]
\newtheorem{proposition}[theorem]{Proposition}
\newtheorem{lemma}[theorem]{Lemma}

\theoremstyle{remark}
\newtheorem{remark}[theorem]{Remark}

\newcommand{\ind}{\mathbf 1}
\newcommand{\e}{\mathrm e}
\newcommand{\F}{\mathcal F}
\DeclareMathOperator{\lcm}{lcm}

\title{Unbounded logarithmic limsup in Erd\H{o}s Problem 684\\
via shifted carry scheduling}
\author{Ji Ho Bae}
\date{September 2026}

\hypersetup{
  pdftitle={Unbounded logarithmic limsup in Erdos Problem 684 via shifted carry scheduling},
  pdfauthor={Ji Ho Bae}
}

\begin{document}
\maketitle

\begin{abstract}
For \(1\leq k\leq n\), let
\[
 u(n,k)=\prod_{p\leq k}p^{\nu_p\binom nk},
 \qquad
 f(n)=\min\{1\leq k\leq n:u(n,k)>n^2\}.
\]
The minimum is interpreted as \(+\infty\) if the set is empty.
Here \(\nu_p(m)\) denotes the exponent of the prime \(p\) in \(m\).
Erd\H{o}s Problem~684 asks for bounds on \(f(n)\).  We prove
\[
 \limsup_{n\to\infty}
 \frac{f(n)}{\log n}
 \frac{\log\log\log n}{\log\log n}\geq \frac12.
\]
In particular,
\(\limsup_{n\to\infty}f(n)/\log n=\infty\), so no uniform estimate
\(f(n)=O(\log n)\) is possible.

The proof constructs integers \(n=tL_M-h-1\), where
\(L_M=\lcm(1,\ldots,M)\).  Product-cell coding yields simultaneous two-sided
approximations to \(tL_M\) modulo every prime in \((M,K]\).  Writing
\(n+1=tL_M-h\), the shift by \(h\) folds both signs into the same one-sided
carry region.  The carries at levels at most \(M\) are bounded by
\(\log\binom{h+k}{h}\).  For the prime powers above \(K\), a truncated CRT
witness has modulus below the search range; exponential weighting then gives
an exponentially small exceptional set.  An elementary anchored-fibre lemma
selects a multiplier satisfying both requirements.  Together, these
ingredients prove the stated bound unconditionally.  The theorem has been
formally verified in Lean~4 with Mathlib.
\end{abstract}



\section{Introduction and statement of the result}

For an integer \(n\geq2\) and \(0\leq k\leq n\), write \(\nu_p(m)\) for
the exponent of the prime \(p\) in \(m\), and define the part of
\(\binom nk\) supported on primes at most \(k\) by
\begin{equation}\label{eq:def-u}
 u(n,k):=\prod_{p\leq k}p^{\nu_p\binom nk}.
\end{equation}
The product is empty, and hence equal to one, when \(k=0\).
The associated threshold is
\begin{equation}\label{eq:def-f}
 f(n):=\min\{1\leq k\leq n:u(n,k)>n^2\}.
\end{equation}
The general upper bound of Alexeev, Putterman, Sawhney, Sellke, and Valiant
\cite{APSSV26} implies that the displayed set is nonempty for all
sufficiently large \(n\); we set \(f(n)=+\infty\) when the defining set is
empty.
Erd\H{o}s introduced this threshold, under the notation \(A(n)\),
in 1979 \cite[pp.~76--77]{Erdos1979}; its modern formulation is recorded as
Problem~684 in Bloom's database \cite{Bloom684}.  Erd\H{o}s observed that
Mahler's theorem implies \(f(n)\to\infty\), ineffectively
\cite{Mahler1953}.

Alexeev, Putterman, Sawhney, Sellke, and Valiant proved the general bound
\begin{equation}\label{eq:known-upper}
 f(n)\leq
 \left(\frac{24}{\pi^2-6}+o(1)\right)(\log n)^2,
\end{equation}
and the same authors constructed a sequence on which
\(f(n)\geq(1/2+o(1))\log n\) \cite[Theorem~2.1]{APSSV26}.
For comparison, Li proved the density-one law
\[
 f(n)=\left(\frac{2}{1-\gamma}+o(1)\right)\log n
\]
for almost all \(n\), where \(\gamma\) is the Euler--Mascheroni constant
\cite{Li2026}.  That result does not constrain the exceptional sequence
studied here.  Erd\H{o}s asked for bounds on \(f(n)\).
Our theorem gives a new quantitative lower bound on its worst-case side,
strengthening the previously known subsequential lower bound by a factor of
order \(\log\log n/\log\log\log n\).  As an immediate consequence it settles
the associated logarithmic-scale question: \(f(n)/\log n\) is unbounded, and
hence no general estimate \(f(n)=O(\log n)\) can hold.  The remaining
quantitative problem is to close the gap between our subsequential lower bound
and the general upper bound \eqref{eq:known-upper}.

\begin{remark}[Relation to arXiv:2604.23784v1--v2]
\label{rem:version-history}
Versions 1 and 2 used a weighted extension of Timofeev's
mean-in-progressions method \cite{Timofeev1997} that was not justified by the
cited theorem; those arguments are withdrawn.  The present version is a
complete rewrite proving the stronger quantitative statement
\eqref{eq:main} by an independent shifted-carry argument.  No Timofeev,
Burgess, or Fourier-sieve input is used in the proof below.
\end{remark}

Our main result is the following.

\begin{theorem}[Quantitative worst-case lower bound for Problem~684]
\label{thm:main}
One has
\begin{equation}\label{eq:main}
 \boxed{
 \limsup_{n\to\infty}
 \frac{f(n)}{\log n}
 \frac{\log\log\log n}{\log\log n}\geq\frac12.}
\end{equation}
Equivalently, there is a sequence \(n_j\to\infty\) such that
\[
 f(n_j)\geq\left(\frac12-o(1)\right)
 \log n_j\frac{\log\log n_j}{\log\log\log n_j}.
\]
Consequently,
\begin{equation}\label{eq:unbounded}
 \limsup_{n\to\infty}\frac{f(n)}{\log n}=\infty.
\end{equation}
Equivalently, for every constant \(C>0\), there are arbitrarily large
integers \(n\) for which \(f(n)>C\log n\).
\end{theorem}

The proof has two selection components.  The arithmetic component bounds the
absolute number of multipliers with a large contribution from high prime
powers.  The
combinatorial component is the following sharp observation: if a finite
interval is coded using \(C\) codewords, and \(\F\) is a set of forbidden
positive differences, then every code fibre whose differences all lie in
\(\F\) has at most \(\lvert\F\rvert+1\) elements.  It is enough to anchor the
fibre at its least point.  This avoids any independence assertion between
simultaneous approximation and the high-power tail.

The rest of the paper is organized as follows.  Section~\ref{sec:prelim}
records the carry formula and the prime-number estimates used in the proof.
Section~\ref{sec:scales} defines the scales and computes the leading exponential
cost of the product-cell code.  Section~\ref{sec:tail} proves the
tail estimate for high prime powers, including the reduced-modulus check needed
before an interval-counting error can be absorbed.  Section~\ref{sec:select}
performs the anchored-fibre selection and the affine fold.
Section~\ref{sec:budget} divides all carries into four disjoint ranges.
Section~\ref{sec:optimization} optimizes the fixed parameters and explains
the order of limits.

Throughout the paper, all logarithms are natural.  Asymptotic notation in a
statement with parameters
\(\lambda,\beta,c,\sigma,\tau\) means that those parameters are held fixed
while the integer \(M\) tends to infinity.  The implicit constants may depend
on these fixed parameters.  Every error term in the carry estimates of
Section~\ref{sec:budget} is uniform over all integers \(1\leq k\leq K\).

\section{Carries and standard prime-number estimates}\label{sec:prelim}

For an integer \(x\) and \(q\geq1\), write \([x]_q\) for the least
nonnegative residue of \(x\) modulo \(q\).  We use
\[
 \vartheta(x)=\sum_{p\leq x}\log p,
 \qquad
 \psi(x)=\sum_{p^a\leq x}\log p,
\]
where \(p\) always denotes a prime and \(a\geq1\).

The next identity is the Legendre--Kummer carry formula
\cite{Kummer1852}.

\begin{lemma}[Legendre--Kummer carry formula]\label{lem:kummer}
For \(0\leq k\leq n\),
\begin{equation}\label{eq:kummer}
 \log u(n,k)
 =\sum_{p\leq k}\sum_{a\geq1}
   \log p\,\ind\{[n]_{p^a}<[k]_{p^a}\}.
\end{equation}
The sum is finite.  In particular, let \(K\) be an integer with
\(1\leq K\leq n\).  If
\begin{equation}\label{eq:criterion}
 \log u(n,k)\leq2\log n\qquad(1\leq k\leq K),
\end{equation}
then \(f(n)>K\).
\end{lemma}

\begin{proof}
For \(q=p^a\), write \(n=qN+r\) and \(k=qK_0+s\), where
\(0\leq r,s<q\).  The \(q\)-summand in Legendre's formula is
\[
 \left\lfloor\frac nq\right\rfloor
 -\left\lfloor\frac kq\right\rfloor
 -\left\lfloor\frac{n-k}{q}\right\rfloor
 =\ind\{r<s\}.
\]
Summing over \(a\geq1\), then over \(p\leq k\), proves
\eqref{eq:kummer}.  If \(p^a>n\), then
\([n]_{p^a}=n\geq k=[k]_{p^a}\), so the corresponding term vanishes.
The last assertion follows directly from the definition of \(f(n)\).
\end{proof}

We shall use the following standard consequences of Chebyshev's bounds and
the prime number theorem with its classical remainder term; see, for example,
\cite[Chapters~I.1 and II.4]{Tenenbaum}.  There is an absolute constant
\(a_0>0\) such that
\begin{equation}\label{eq:pnt-error}
 \vartheta(x)=x+O\!\left(x\e^{-a_0\sqrt{\log x}}\right)
 \qquad(x\geq2),
\end{equation}
and
\begin{equation}\label{eq:psi-theta}
 \psi(x)-\vartheta(x)=O(\sqrt{x}),
 \qquad
 \psi(x)=x+o(x).
\end{equation}
Indeed,
\[
 \psi(x)-\vartheta(x)
 =\sum_{a\geq2}\vartheta(x^{1/a})
 =O(\sqrt{x})+O(x^{1/3}\log x)=O(\sqrt{x})
\]
by Chebyshev's estimate \(\vartheta(y)\ll y\).
The analogous estimate holds for the prime-counting function \(\pi(x)\).
Partial summation also gives, whenever
\(A\to\infty\) and \(\log A=o(\log M)\),
\begin{equation}\label{eq:prime-harmonic}
 \sum_{M<p\leq AM}\frac1p
 =O\!\left(\frac{\log A}{\log M}\right).
\end{equation}
Only these classical estimates, the carry formula, and the ordinary Chinese
remainder theorem are used.

\section{Scales and the product-cell code}\label{sec:scales}

Fix constants satisfying
\begin{equation}\label{eq:param-basic}
 0<\lambda<1,\qquad \beta>0,\qquad
 0<c<\lambda\beta,\qquad \sigma>2\beta,
\end{equation}
and fix a real number \(\tau\) such that
\begin{equation}\label{eq:param-tau}
 \tau<\sigma-c,
 \qquad
 \beta+\frac3{20}<2(1+\tau).
\end{equation}
The second inequality implies \(1+\tau>0\).  For a sufficiently large
integer \(M\), put
\begin{equation}\label{eq:scales}
 \begin{aligned}
 A&=c\frac{\log M}{\log\log M},
 &K&=\lfloor AM\rfloor,\\
 h&=\left\lfloor\frac{M}{20\log A}\right\rfloor,
 &J&=\lfloor\e^{\sigma M}\rfloor,\\
 L&=L_M=\lcm(1,2,\ldots,M).
 \end{aligned}
\end{equation}
Thus
\begin{equation}\label{eq:scale-facts}
 A\to\infty,\quad 2h<M<K,\quad K=o(M^2),\quad
 \log L=\psi(M)=M+o(M).
\end{equation}

For each prime \(M<p\leq K\), partition the integer residues in
\([0,p)\) into the half-open intervals
\([rh,(r+1)h)\cap[0,p)\).  There are
\(\lceil p/h\rceil\) such cells, and two integer residues in the same cell
differ in absolute value by less than \(h\).  Code each
\(j\in\{0,1,\ldots,J\}\) by the product cell containing
\begin{equation}\label{eq:code}
 \bigl([jL]_p\bigr)_{M<p\leq K}.
\end{equation}
The number of available codewords is at most
\begin{equation}\label{eq:def-CM}
 C_M:=\prod_{M<p\leq K}\left\lceil\frac ph\right\rceil.
\end{equation}

The exponent of this coding cost is needed with its leading constant.

\begin{lemma}[Code entropy]\label{lem:entropy}
For the scales in \eqref{eq:scales},
\begin{equation}\label{eq:entropy}
 \log C_M=(c+o(1))M.
\end{equation}
\end{lemma}

\begin{proof}
Uniformly for \(M<p\leq K\), one has \(p/h>20\log A\), and
\(20\log A\to\infty\).  Since
\(0\leq\log\lceil x\rceil-\log x\leq1/x\) for \(x\geq1\),
\eqref{eq:prime-harmonic} yields
\begin{equation}\label{eq:ceiling-error}
 \sum_{M<p\leq K}
 \left(\log\left\lceil\frac ph\right\rceil-\log\frac ph\right)
 \leq h\sum_{M<p\leq K}\frac1p=o(M).
\end{equation}
It remains to evaluate
\(S_M=\sum_{M<p\leq K}\log(p/h)\).

By partial summation and the quantitative remainder
\eqref{eq:pnt-error},
\begin{equation}\label{eq:entropy-integral}
 S_M=\int_M^K\frac{\log(x/h)}{\log x}\,dx+o(M).
\end{equation}
Indeed, the total PNT-remainder contribution is
\[
 O\!\left(K\log A\,\e^{-a_1\sqrt{\log M}}\right)=o(M)
\]
for some absolute \(a_1>0\).  On writing \(x=My\), the integral in
\eqref{eq:entropy-integral} becomes
\begin{equation}\label{eq:entropy-scaled}
 M\int_1^A
 \frac{\log(20y\log A)+o(1)}{\log M+\log y}\,dy+o(M).
\end{equation}
The floor in \(h\) accounts for the uniform \(o(1)\).  Replacing the
denominator by \(\log M\) costs at most
\[
 O\!\left(\frac{MA(\log A)^2}{(\log M)^2}\right)=o(M).
\]
Finally,
\[
 \int_1^A\log(20y\log A)\,dy
 =A\log(20A\log A)-A+O(\log\log A)
 =(c+o(1))\log M.
\]
Together with \eqref{eq:ceiling-error}, this proves
\eqref{eq:entropy}.
\end{proof}

Equal codewords give the required simultaneous two-sided approximation.
For an integer \(x\), let \(\|x\|_p\) denote the absolute value of the
representative of \(x\pmod p\) in \((-p/2,p/2]\).

\begin{lemma}[Equal-cell difference]\label{lem:equal-cell}
If \(0\leq i<j\leq J\) have the same code and \(t=j-i\), then
\begin{equation}\label{eq:balanced}
 \|tL\|_p<h\qquad(M<p\leq K).
\end{equation}
\end{lemma}

\begin{proof}
For each indicated prime, the ordinary integer difference
\([jL]_p-[iL]_p\) has absolute value less than \(h\) and represents
\(tL\pmod p\).  Since \(2h<M<p\), it is the signed representative.
\end{proof}

\section{An exponential estimate for the high-power tail}\label{sec:tail}

For \(1\leq t\leq J\), set
\begin{equation}\label{eq:def-nt}
 n_t=tL-h-1
\end{equation}
and define
\begin{equation}\label{eq:def-Z}
 Z(t)=
 \sum_{\substack{p\leq K,\ a\geq2\\p^a>K}}
 \log p\,\ind\{[n_t]_{p^a}<K\}.
\end{equation}
For all sufficiently large \(M\), \(n_t>K\) for every \(t\geq1\).
Moreover, if \(p^a>n_t\), then
\([n_t]_{p^a}=n_t>K\).  Thus \eqref{eq:def-Z} is a finite sum despite
its notation.

\begin{lemma}[High-power tail]\label{lem:tail}
For every fixed parameter tuple satisfying \eqref{eq:param-basic},
\begin{equation}\label{eq:tail-bound}
 \#\{1\leq t\leq J:Z(t)\geq\beta M\}
 \leq
 2J\exp\!\left(-\lambda\beta M
       +O_\lambda\bigl(K^{(1+\lambda)/2}\bigr)\right).
\end{equation}
In particular, because \(\lambda<1\) is fixed and
\(K=M\log^{1+o(1)}M\), the error in the exponential is \(o(M)\).
\end{lemma}

\begin{proof}
We divide the proof into four steps.

\medskip
\noindent\emph{Step 1: nested towers.}
For \(p\leq K\), define
\begin{equation}\label{eq:b-a}
 b_p=\nu_p(L),
 \qquad
 a_p=\min\{a\geq2:p^a>K\}.
\end{equation}
Since \(p^{b_p}\leq M<K<p^{a_p}\), one has \(a_p>b_p\).
For \(a\geq a_p\), the events
\[
 E_{p,a}(t):=\{[n_t]_{p^a}<K\}
\]
are nested downwards in \(a\).  Indeed, if \(E_{p,a+1}(t)\) holds, its
residue is less than \(K<p^a\), so reduction modulo \(p^a\) leaves the
same residue and \(E_{p,a}(t)\) holds.  Let \(\ell_p(t)\) be the length of
the initial true segment beginning at \(a_p\), with length zero if the first
event fails.  The preceding finiteness observation gives
\begin{equation}\label{eq:Z-depth}
 Z(t)=\sum_{p\leq K}\ell_p(t)\log p.
\end{equation}

\medskip
\noindent\emph{Step 2: congruence classes for a depth profile.}
Put \(g_p=p^{b_p}\).  If \(\ell_p(t)\geq r\geq1\), then for some integer
\(0\leq s<K\),
\begin{equation}\label{eq:depth-congruence}
 tL\equiv h+1+s\pmod {p^{a_p+r-1}}.
\end{equation}
Solvability forces \(g_p\mid h+1+s\).  There are at most
\(K/g_p+1\leq2K/g_p\) possible values of \(s\), since
\(g_p\leq M<K\).  After division by \(g_p\), the integer \(L/g_p\) is a
unit modulo
\begin{equation}\label{eq:Qpr}
 Q_p(r)=p^{a_p+r-1-b_p}.
\end{equation}
Consequently, the condition \(\ell_p(t)\geq r\) restricts \(t\) to at
most \(2K/g_p\) residue classes modulo \(Q_p(r)\).

For a finitely supported profile
\(\mathbf r=(r_p)_{p\leq K}\) of nonnegative integers, define
\begin{equation}\label{eq:RQ}
 R(\mathbf r)=\prod_{r_p>0}\frac{2K}{g_p},
 \qquad
 Q(\mathbf r)=\prod_{r_p>0}Q_p(r_p),
 \qquad
 W(\mathbf r)=\sum_{p\leq K}r_p\log p.
\end{equation}
The moduli \(Q_p(r_p)\) have distinct prime bases.  The Chinese remainder
theorem and elementary interval counting therefore give
\begin{equation}\label{eq:crt-plus-one}
 \#\{1\leq t\leq J:\ell_p(t)\geq r_p\text{ for every }p\}
 \leq R(\mathbf r)\left(\frac{J}{Q(\mathbf r)}+1\right).
\end{equation}

\medskip
\noindent\emph{Step 3: a short-modulus witness for every bad multiplier.}
Suppose that \(Z(t)\geq\beta M\).  Starting from the zero profile, add
one unit at a time to coordinates without exceeding
\(\ell_p(t)\), and stop when the weight first reaches \(\beta M\).
Since an increment has weight at most \(\log K\), the resulting prefix
profile satisfies
\begin{equation}\label{eq:minimal-prefix}
 \beta M\leq W(\mathbf r)<\beta M+\log K.
\end{equation}
Let \(\mathcal W_M\) be the finite set of all nonnegative profiles satisfying
\eqref{eq:minimal-prefix}.  Every multiplier with \(Z(t)\geq\beta M\)
belongs to the depth event associated with at least one profile in
\(\mathcal W_M\).  We now prove that every profile in \(\mathcal W_M\) has
reduced CRT modulus below the search length \(J\).

From \eqref{eq:Qpr} and \eqref{eq:RQ},
\begin{equation}\label{eq:logQ}
 \log Q(\mathbf r)
 =W(\mathbf r)+
  \sum_{r_p>0}(a_p-1-b_p)\log p.
\end{equation}
For \(p>M\), the relation \(K<M^2<p^2\) gives
\(a_p=2\) and \(b_p=0\); hence that prime's summand in the second term
is \(\log p\leq r_p\log p\).  For \(p\leq M\), one has
\[
 a_p-1=\max\{a:p^a\leq K\},
 \qquad
 b_p=\max\{b:p^b\leq M\}.
\]
Thus, on summing even over all primes rather than merely the support of
\(\mathbf r\), cancellation of the first-power terms gives the exact identity
\begin{equation}\label{eq:overhead-identity}
 \sum_{p\leq M}(a_p-1-b_p)\log p
 =\{\psi(K)-\vartheta(K)\}
  -\{\psi(M)-\vartheta(M)\}.
\end{equation}
Here \(K<M^2\) ensures that every prime having a square at most \(K\) is
at most \(M\).  The expression in \eqref{eq:overhead-identity} is
nonnegative and is \(O(\sqrt K)=o(M)\) by
\eqref{eq:psi-theta}.  It follows from
\eqref{eq:minimal-prefix}--\eqref{eq:overhead-identity} that
\begin{equation}\label{eq:Q-before-plus-one}
 \log Q(\mathbf r)
 \leq2W(\mathbf r)+O(\sqrt K)
 \leq2\beta M+o(M).
\end{equation}
Put \(\varepsilon_0=(\sigma-2\beta)/3>0\).  Since
\(\log J=\sigma M+o(1)\), for all sufficiently large \(M\) we have
\[
 \log Q(\mathbf r)
 \leq(2\beta+\varepsilon_0)M
 <(\sigma-\varepsilon_0)M
 <\log J.
\]
Thus
\begin{equation}\label{eq:Q-less-J}
 Q(\mathbf r)<J
\end{equation}
for every \(\mathbf r\in\mathcal W_M\).

For every \(\mathbf r\in\mathcal W_M\), \eqref{eq:Q-less-J} permits
\(J/Q(\mathbf r)+1\leq2J/Q(\mathbf r)\) in
\eqref{eq:crt-plus-one}.  The factors
\(g_p\) cancel exactly, and \eqref{eq:crt-plus-one} becomes
\begin{equation}\label{eq:profile-count}
 \#\{1\leq t\leq J:\ell_p(t)\geq r_p\text{ for every }p\}
 \leq
 2J\prod_{r_p>0}\frac{2K}{p^{a_p+r_p-1}}.
\end{equation}
This also covers the original tail containing individual prime powers larger
than \(J\): only the reduced joint modulus of the extracted prefix is used,
and \eqref{eq:Q-less-J} places that modulus in the legitimate range.

\medskip
\noindent\emph{Step 4: exponential weighting.}
Take the union in \eqref{eq:profile-count} over
\(\mathbf r\in\mathcal W_M\).  Apply exponential weighting in the form
\(\ind\{W\geq\beta M\}\leq
\e^{-\lambda\beta M}\e^{\lambda W}\).  Enlarging the resulting
\emph{nonnegative numerical sum} to all profiles gives
\begin{align}
 \#\{1\leq t\leq J:Z(t)\geq\beta M\}
 &\leq2J\e^{-\lambda\beta M}
 \prod_{p\leq K}
 \left(1+\sum_{r\geq1}
   \frac{2Kp^{\lambda r}}{p^{a_p+r-1}}\right)\notag\\
 &=2J\e^{-\lambda\beta M}
 \prod_{p\leq K}
 \left(1+
   \frac{2K}{p^{a_p-1}(p^{1-\lambda}-1)}\right).
 \label{eq:rankin-product}
\end{align}
The last enlargement is purely numerical: the CRT estimate has already been
applied to every profile in \(\mathcal W_M\).

For fixed \(\lambda<1\), one has
\(p^{1-\lambda}-1\gg_\lambda p^{1-\lambda}\) uniformly for \(p\geq2\).
Consequently, the logarithm of the product in
\eqref{eq:rankin-product} is at most a constant depending on \(\lambda\)
times
\[
 \sum_{p\leq\sqrt K}p^\lambda
 +K\sum_{p>\sqrt K}p^{-(2-\lambda)}.
\]
Indeed, for \(p\leq\sqrt K\), the minimality of \(a_p\) gives
\(p^{a_p-1}\leq K<p^{a_p}\), whereas for \(p>\sqrt K\) one has
\(a_p=2\).  Enlarging the prime sums to integer sums shows that the last
display is
\begin{equation}\label{eq:euler-error}
 O_\lambda\bigl(K^{(1+\lambda)/2}\bigr).
\end{equation}
Using \(\log(1+x)\leq x\) proves \eqref{eq:tail-bound}.
Finally,
\[
 \frac{K^{(1+\lambda)/2}}M
 =(1+o(1))M^{-(1-\lambda)/2}
  \left(c\frac{\log M}{\log\log M}\right)^{(1+\lambda)/2}
 =o(1)
\]
for every fixed \(\lambda<1\).
\end{proof}

\section{Anchored-fibre selection and the affine fold}\label{sec:select}

We isolate the combinatorial selection principle.

\begin{lemma}[Anchored code-fibre avoidance]\label{lem:anchor}
Let \(\Phi:\{0,1,\ldots,J\}\to\Omega\), where
\(\lvert\Omega\rvert\leq C\), and let
\(\F\subseteq\{1,\ldots,J\}\).  If
\begin{equation}\label{eq:anchor-condition}
 J+1>C(\lvert\F\rvert+1),
\end{equation}
then there are \(0\leq i<j\leq J\) such that
\begin{equation}\label{eq:anchor-conclusion}
 \Phi(i)=\Phi(j),\qquad j-i\notin\F.
\end{equation}
\end{lemma}

\begin{proof}
Some fibre contains more than \(\lvert\F\rvert+1\) elements.  Anchor that
fibre at its least element \(i_0\).  As \(j\) runs through the other
elements of the fibre, the positive differences \(j-i_0\) are distinct.
At most \(\lvert\F\rvert\) of them lie in \(\F\), so at least one does not.
\end{proof}

Define the forbidden set of positive differences by
\begin{equation}\label{eq:forbidden}
 \F_M=
 \{1\leq t\leq J:t\leq\e^{\tau M}\}
 \mathbin{\cup}
 \{1\leq t\leq J:Z(t)\geq\beta M\}.
\end{equation}
We apply Lemma~\ref{lem:anchor} to the product-cell code
\eqref{eq:code}.  Lemmas~\ref{lem:entropy} and~\ref{lem:tail} give
\begin{align}
 \frac{C_M(\lvert\F_M\rvert+1)}{J+1}
 &\ll
 \exp\bigl(-(\sigma-c-\tau+o(1))M\bigr)\notag\\
 &\quad+
 \exp\bigl(-(\lambda\beta-c+o(1))M\bigr)
 +\exp\bigl(-(\sigma-c+o(1))M\bigr).
 \label{eq:forbidden-ratio}
\end{align}
All three terms tend to zero.  The first two exponent gaps are positive by
\eqref{eq:param-basic}--\eqref{eq:param-tau}; the third is positive because
\(c<\lambda\beta<\beta<\sigma\).  Thus
\eqref{eq:anchor-condition} holds for all sufficiently large \(M\).
Lemmas~\ref{lem:anchor} and~\ref{lem:equal-cell} provide an integer
\(t=j-i\) such that
\begin{equation}\label{eq:selected-t}
 \e^{\tau M}<t\leq\e^{\sigma M},
 \qquad
 Z(t)<\beta M,
 \qquad
 \|tL\|_p<h\quad(M<p\leq K).
\end{equation}
The tail predicate depends only on the difference \(t\), not on the pair
\((i,j)\); no probabilistic independence is used.

Fix such a \(t\), and henceforth put
\begin{equation}\label{eq:def-n}
 n=tL-h-1.
\end{equation}
For \(M<p\leq K\), let \(z_p\in(-h,h)\cap\mathbb Z\) be the signed
representative of \(tL\pmod p\).  The common shift gives
\begin{equation}\label{eq:affine-fold}
 n+1\equiv z_p-h=-d_p\pmod p,
 \qquad
 d_p=h-z_p\in\{1,2,\ldots,2h-1\}.
\end{equation}
Thus both signs of the symmetric approximation have been folded into the
same final \(2h-1\) residue positions of every \(p\)-block.

\section{The four carry ranges}\label{sec:budget}

Write
\begin{equation}\label{eq:Dnk}
 D_n(k)=\log u(n,k).
\end{equation}
We prove, uniformly for every integer \(1\leq k\leq K\), that
\begin{equation}\label{eq:budget-goal}
 D_n(k)\leq\left(\beta+\frac3{20}+o(1)\right)M.
\end{equation}
The prime-power levels in Lemma~\ref{lem:kummer} are divided into the four
disjoint ranges
\[
 \begin{array}{ll}
 \text{(I)}&q=p^a\leq M,\\
 \text{(II)}&q=p,\quad M<p\leq K,\\
 \text{(III)}&M<q=p^a\leq K,\quad a\geq2,\\
 \text{(IV)}&q=p^a>K.
 \end{array}
\]
In every range the additional restriction \(p\leq k\) from
\eqref{eq:kummer} remains in force.

\subsection{\texorpdfstring{Levels at most \(M\)}{Levels at most M}}

\begin{lemma}[Coherent small-level cost]\label{lem:small-level}
The contribution of range \emph{(I)} is at most
\begin{equation}\label{eq:small-binomial}
 \log\binom{h+k}{h}
 \leq\left(\frac1{20}+o(1)\right)M,
\end{equation}
uniformly for \(k\leq K\).
\end{lemma}

\begin{proof}
If \(q=p^a\leq M\), then \(q\mid L\), and hence
\begin{equation}\label{eq:small-residue}
 [n]_q=[-h-1]_q=q-1-[h]_q.
\end{equation}
This formula includes the case \([h]_q=q-1\).  Therefore
\begin{equation}\label{eq:carry-identity}
 [n]_q<[k]_q
 \quad\Longleftrightarrow\quad
 [h]_q+[k]_q\geq q.
\end{equation}
The indicator on the right is the \(q=p^a\) carry summand in
\(\nu_p\binom{h+k}{h}\).  Range~\emph{(I)}, which additionally restricts to
\(p\leq k\) and \(q\leq M\), is consequently a subsum of the nonnegative
prime-power expansion of \(\log\binom{h+k}{h}\).

The elementary estimate
\[
 \binom{h+k}{h}\leq
 \left(\e\left(1+\frac{k}{h}\right)\right)^h
\]
gives, uniformly for \(k\leq K\),
\begin{align*}
 \log\binom{h+k}{h}
 &\leq h\left(1+\log\left(1+\frac Kh\right)\right)\\
 &=\frac{M}{20\log A}
   \bigl(\log A+O(\log\log A)\bigr)+o(M)
 =\left(\frac1{20}+o(1)\right)M.
\end{align*}
\end{proof}

\subsection{\texorpdfstring{First prime levels above \(M\)}{First prime levels above M}}

\begin{lemma}[Scheduled first-level cost]\label{lem:first-level}
The contribution of range \emph{(II)} is at most
\begin{equation}\label{eq:first-level-bound}
 \left(\frac1{10}+o(1)\right)M
\end{equation}
uniformly for \(k\leq K\).
\end{lemma}

\begin{proof}
By \eqref{eq:affine-fold},
\begin{equation}\label{eq:nmodp}
 [n]_p=p-1-d_p,
 \qquad 1\leq d_p<2h.
\end{equation}
Suppose \(M<p\leq k\) contributes, and set
\(m=\lfloor k/p\rfloor+1\).  The inequality
\([n]_p<[k]_p\) and \eqref{eq:nmodp} imply
\[
 1\leq mp-k=p-[k]_p\leq d_p<2h.
\]
Consequently,
\begin{equation}\label{eq:active-interval}
 p\in I_m(k):=
 \left(\frac{k}{m},\frac{k+2h}{m}\right],
 \qquad 2\leq m\leq A+1.
\end{equation}

Let \(\mathcal M(k)\) consist only of those indices \(m\) for which
\(I_m(k)\) contains at least one contributing prime in \((M,k]\).
This restriction prevents a spurious use of a relative short-interval
estimate near zero.  If \(m\in\mathcal M(k)\), then the presence of such a
prime implies
\begin{equation}\label{eq:active-scale}
 \frac{k}{m}>M-\frac{2h}{m}\geq M-h=M(1-o(1)),
\end{equation}
whereas
\begin{equation}\label{eq:min-interval-length}
 |I_m(k)|=\frac{2h}{m}\geq\frac{2h}{A+1}
 \asymp\frac{M}{\log M}.
\end{equation}
All upper endpoints are at most \(K+2h\).  It follows from
\eqref{eq:pnt-error}, uniformly in \(k\) and in every active \(m\), that
\begin{equation}\label{eq:interval-pnt}
 \vartheta\!\left(\frac{k+2h}{m}\right)
 -\vartheta\!\left(\frac{k}{m}\right)
 \leq\frac{2h}{m}
 +O\!\left(K\e^{-a_2\sqrt{\log M}}\right)
\end{equation}
for an absolute \(a_2>0\).  There are at most \(A+1\) active indices, and
\(AK\e^{-a_2\sqrt{\log M}}=o(M)\).  Hence, allowing overlap between
intervals (which can only enlarge an upper bound),
\begin{align}
 \sum_{\substack{M<p\leq k\\ {}[n]_p<[k]_p}}\log p
 &\leq
 \sum_{m\in\mathcal M(k)}
 \left\{\vartheta\!\left(\frac{k+2h}{m}\right)
       -\vartheta\!\left(\frac{k}{m}\right)\right\}\notag\\
 &\leq2h\sum_{2\leq m\leq A+1}\frac1m+o(M)
 =\left(\frac1{10}+o(1)\right)M.
 \label{eq:first-level-sum}
\end{align}
The estimate is also valid when the range is empty, so it is uniform for all
\(1\leq k\leq K\).
\end{proof}

\subsection{Intermediate higher powers}

The contribution of range~\emph{(III)} is bounded without using residue
information:
\begin{equation}\label{eq:intermediate}
 \sum_{\substack{M<p^a\leq K\\a\geq2}}\log p
 =\{\psi(K)-\vartheta(K)\}
  -\{\psi(M)-\vartheta(M)\}
 =O(\sqrt K)=o(M).
\end{equation}

\subsection{\texorpdfstring{The tail above \(K\)}{The tail above K}}

If \(q=p^a>K\) and \(p\leq k\leq K\), then necessarily \(a\geq2\) and
\([k]_q=k\).  Thus the contribution of range~\emph{(IV)} is at most
\begin{equation}\label{eq:tail-budget}
 \sum_{\substack{p\leq k,\ a\geq2\\p^a>K}}
 \log p\,\ind\{[n]_{p^a}<k\}
 \leq Z(t)<\beta M.
\end{equation}
Terms with \(p^a>n\) vanish, as noted following \eqref{eq:def-Z}.

Adding \eqref{eq:small-binomial}, \eqref{eq:first-level-bound},
\eqref{eq:intermediate}, and \eqref{eq:tail-budget} proves
\eqref{eq:budget-goal}, with a single \(o(M)\) term that is uniform over every
integer \(1\leq k\leq K\).

\section{Completion and optimization}\label{sec:optimization}

First we record exactly what the fixed parameters give.

\begin{proposition}[Fixed-parameter construction]\label{prop:fixed}
Let \(\lambda,\beta,c,\sigma,\tau\) be fixed and satisfy
\eqref{eq:param-basic}--\eqref{eq:param-tau}.  For every sufficiently large
integer \(M\), there is an integer \(n=n_M\) such that
\begin{equation}\label{eq:fixed-conclusion}
 f(n)>K=\left\lfloor cM\frac{\log M}{\log\log M}\right\rfloor
\end{equation}
and
\begin{equation}\label{eq:n-size}
 (1+\tau+o(1))M\leq\log n\leq(1+\sigma+o(1))M.
\end{equation}
Consequently,
\begin{equation}\label{eq:fixed-normalized}
 \limsup_{n\to\infty}
 \frac{f(n)}{\log n}
 \frac{\log\log\log n}{\log\log n}
 \geq\frac{c}{1+\sigma}.
\end{equation}
\end{proposition}

\begin{proof}
Choose \(t\) by \eqref{eq:selected-t} and \(n\) by
\eqref{eq:def-n}.  Since \(\log L=\psi(M)=M+o(M)\) and
\((h+1)/(tL)\to0\),
\begin{equation}\label{eq:log-n-exact}
 \log n=\log t+\log L+o(1).
\end{equation}
The bounds on \(t\) in \eqref{eq:selected-t} imply
\eqref{eq:n-size}.  In particular, the lower size bound and
\eqref{eq:param-tau} give
\begin{equation}\label{eq:strict-threshold}
 \left(\beta+\frac3{20}+o(1)\right)M
 <2\log n
\end{equation}
for all sufficiently large \(M\).  For those \(M\), the observation following
\eqref{eq:def-Z} also gives \(K<n\).  Equations
\eqref{eq:budget-goal} and \eqref{eq:strict-threshold}, together with
Lemma~\ref{lem:kummer}, show that \(f(n)>K\), proving
\eqref{eq:fixed-conclusion}.

The upper size bound in \eqref{eq:n-size} yields
\begin{equation}\label{eq:ratio-M}
 \frac{f(n)}{\log n}
 \geq
 \left(\frac{c}{1+\sigma}+o(1)\right)
 \frac{\log M}{\log\log M}.
\end{equation}
For fixed parameters, \(\log n=\Theta(M)\), so
\begin{equation}\label{eq:iterated-log-conversion}
 \log\log n=\log M+O(1),
 \qquad
 \log\log\log n=\log\log M+o(1).
\end{equation}
Multiplying \eqref{eq:ratio-M} by
\(\log\log\log n/\log\log n\) proves
\eqref{eq:fixed-normalized}.  Finally,
\(n\geq L-h-1\to\infty\), so the constructed integers form an unbounded
sequence (after passage to a subsequence if necessary).
\end{proof}

\begin{proof}[Proof of Theorem~\ref{thm:main}]
Fix \(0<\delta<1/2\) and \(\beta>0\), and choose
\begin{equation}\label{eq:optimized-parameters}
 \lambda=1-\delta,\qquad
 c=(1-2\delta)\beta,\qquad
 \sigma=(2+\delta)\beta,\qquad
 \tau=\beta.
\end{equation}
All required inequalities have explicit positive margins:
\begin{align}
 \lambda\beta-c&=\delta\beta,
 &\sigma-2\beta&=\delta\beta,\notag\\
 \sigma-c-\tau&=3\delta\beta,
 &2(1+\tau)-\left(\beta+\frac3{20}\right)
 &=\beta+\frac{37}{20}.
 \label{eq:parameter-margins}
\end{align}
Proposition~\ref{prop:fixed} therefore gives
\begin{equation}\label{eq:optimized-scalar}
 \limsup_{n\to\infty}
 \frac{f(n)}{\log n}
 \frac{\log\log\log n}{\log\log n}
 \geq
 \frac{(1-2\delta)\beta}{1+(2+\delta)\beta}.
\end{equation}

The order of limits is as follows.  First fix \(\delta\) and \(\beta\),
and let \(M\to\infty\) in Proposition~\ref{prop:fixed}; only after that
limit is taken do we optimize the scalar right side of
\eqref{eq:optimized-scalar}.  Thus no constant in
\(O_\lambda(\cdot)\) and no \(o(M)\) term is required to be uniform as
\(\lambda\uparrow1\) or \(\beta\to\infty\).  Taking the supremum first
over fixed \(\beta>0\) and then over fixed \(\delta>0\) gives
\[
 \sup_{0<\delta<1/2}\sup_{\beta>0}
 \frac{(1-2\delta)\beta}{1+(2+\delta)\beta}
 =\sup_{0<\delta<1/2}\frac{1-2\delta}{2+\delta}
 =\frac12.
\]
This proves \eqref{eq:main}.

Equivalently, one may form a literal diagonal sequence: choose
\(\delta_j\downarrow0\) and \(\beta_j\to\infty\), and then, only after
fixing the pair \((\delta_j,\beta_j)\), choose \(M_j\) sufficiently large
for all fixed-parameter estimates, for
\(\log\beta_j/\log M_j<1/j\), and for \(n_j>n_{j-1}\).  The normalized
coefficient along this sequence tends to \(1/2\).

Finally, the factor
\(\log\log n/\log\log\log n\) tends to infinity.  Hence
\eqref{eq:main} implies \eqref{eq:unbounded}.
\end{proof}

\section{Structure of the argument}

The product-cell map in \eqref{eq:code} is not a homomorphism, and its
fibres need not be cosets.  The sampled orbit is also much shorter than its
full period: if
\(P=\prod_{M<p\leq K}p\), then
\[
 \log P=\vartheta(K)-\vartheta(M)=(1+o(1))AM,
 \qquad
 \log J=\sigma M+o(M),
\]
so \(J/P\to0\).  Accordingly, Lemma~\ref{lem:anchor} treats the product-cell
map as an arbitrary code and uses no coset structure.

The common shift in \eqref{eq:def-n} has two simultaneous effects.  It
orients all signed approximation errors by \eqref{eq:affine-fold}, and it
creates the small-level carry pattern \eqref{eq:carry-identity}.  The latter
pattern is affordable precisely because its total weight is a subsum of one
binomial coefficient.  Exact annihilation of every level below \(M\) is not
needed.

For the upper tail, powers of a fixed prime are encoded by the nested depth
\(\ell_p(t)\), after which the CRT combines the distinct primes.  The
extracted prefix satisfies \(Q(\mathbf r)<J\) before the interval discrepancy
in \eqref{eq:crt-plus-one} is replaced by a factor of two.

\section{Relation to Erd\H{o}s's original question}

In 1979, after defining \(A(n)\) to be the least \(k\) for which
\(u_k^{(n)}>n^2\), Erd\H{o}s observed that Mahler's theorem gives
\(A(n)\to\infty\), but wrote that ``we do not know how fast.''  He added that
Baker's results might yield a crude estimate \cite[pp.~76--77]{Erdos1979}.  With
the modern notation in \eqref{eq:def-u}--\eqref{eq:def-f}, his \(A(n)\) is
precisely the function \(f(n)\) studied here.

Theorem~\ref{thm:main} gives a direct quantitative answer on the worst-case
side of that question: it supplies an explicit subsequential lower bound,
proves that \(f(n)/\log n\) is unbounded, and hence rules out a uniform
logarithmic upper bound.  Together with the known general estimate
\eqref{eq:known-upper}, this answers Erd\H{o}s's original question by giving
explicit bounds for his threshold.  Determining its exact extremal order
remains a separate open quantitative problem.

\section{Formal verification}\label{sec:formal}

Theorem~\ref{thm:main} has been formalized and machine-checked in Lean~4
\cite{Lean4} with the Mathlib library \cite{Mathlib}.  The development is
available at
\begin{center}
\url{https://github.com/jidodat/erdos684-lean}\qquad(tag \texttt{v3-arxiv}).
\end{center}
It follows the paper section by section: the definitions
\eqref{eq:def-u}--\eqref{eq:def-f}, the carry formula of Lemma~\ref{lem:kummer},
the code of Section~\ref{sec:scales}, the CRT count and exponential weighting of
Lemma~\ref{lem:tail}, the anchored-fibre selection of
Section~\ref{sec:select}, the four carry ranges of Section~\ref{sec:budget},
and the optimization of Section~\ref{sec:optimization} are all proved from
Mathlib alone.  The prime number theorem is the single external input.  The
file \texttt{Main.lean} proves the theorem under the explicit hypothesis
\(\vartheta(x)=x+O(x/\log^2x)\), which is weaker than \eqref{eq:pnt-error},
and the file \texttt{PNT.lean} derives this hypothesis from the bound
\(\psi(x)=x+O\bigl(x\exp(-c(\log x)^{1/10})\bigr)\) established in the
\emph{PrimeNumberTheoremAnd} project \cite{PNTAnd}.  Consequently the final
statements \texttt{main\_theorem\_unconditional} (the form of
Theorem~\ref{thm:main} with a sequence \(n_j\to\infty\)),
\texttt{main\_theorem\_limsup} (the \(\limsup\) form \eqref{eq:main}, with
values in the extended nonnegative reals), and
\texttt{unbounded\_unconditional} (statement \eqref{eq:unbounded}) carry no
hypotheses.  For every theorem of the development, Lean's
\texttt{\#print axioms} reports only the three standard axioms
\texttt{propext}, \texttt{Classical.choice}, and \texttt{Quot.sound}.

The formalization differs from the text in a few simplifications, which are
recorded in the repository.  In Lemma~\ref{lem:entropy} only the upper bound
is needed; it is obtained from \(\vartheta(K)\leq(1+o(1))K\) together with
the monotonicity of \(x\mapsto1-\log h/\log x\) in place of partial summation,
so that \eqref{eq:prime-harmonic} is not used.  The estimate
\(\psi(x)-\vartheta(x)=O(\sqrt x)\) is replaced by Mathlib's
\(\psi(x)-\vartheta(x)\leq2\sqrt x\log x\), which is still \(o(M)\).  The
prefix extraction in Step~3 of Lemma~\ref{lem:tail} is an abstract lemma
proved by induction on the total depth, without ordering the primes.  The
decomposition into the ranges (I)--(IV) carries the hypothesis \(M\leq K\),
which is implicit in \eqref{eq:scale-facts}.  Finally, the definitions of
\(u(n,k)\) and \(f(n)\) are checked against the text by proving that
\(u(n,k)\) times the part of \(\binom nk\) supported on primes exceeding
\(k\) equals \(\binom nk\), and that \(f(n)>K\) holds exactly when
\(u(n,k)\leq n^2\) for all \(1\leq k\leq\min(K,n)\).

\end{document}